\documentclass[12pt]{amsart} 
\usepackage[english]{babel}
\usepackage{amsfonts}
\usepackage[version=3]{mhchem}
\usepackage{amsthm}
\usepackage{amssymb}
\usepackage[top=1in, bottom=1in, left=1in, right=1in]{geometry}
\usepackage{multirow}
\usepackage[table,xcdraw]{xcolor}
\usepackage{array}
\usepackage{tabu}
\usepackage{caption} 
\usepackage[labelfont=bf]{caption} 
\usepackage{siunitx}
\usepackage{tikz}
\usetikzlibrary{patterns}
\usetikzlibrary{arrows}
\usetikzlibrary{calc,shapes.geometric}
\usepackage{float}
\usepackage{amsmath,amsfonts}
\usepackage{multicol}
\usepackage{caption}
\usepackage{subcaption}

\newcommand{\code}{\mathcal{C}}
 
\usepackage[foot]{amsaddr}
 
\title[Neural codes with three maximal codewords]{Neural codes with three maximal codewords: \\ convexity and minimal embedding dimension}
\author{Katherine Johnston$^1$}
\address{$^1$Lafayette College}

\author{Anne Shiu$^2$}
\address{$^2$Texas A\&M University}

\author{Clare Spinner$^3$}
\address{$^3$University of Portland}

\date{August 30,  2020}

\newtheorem{theorem}{Theorem}[section]
\newtheorem{thm}[theorem]{Theorem}
\newtheorem{prop}[theorem]{Proposition}

\newtheorem{lem}[theorem]{Lemma}

\theoremstyle{definition}
\newtheorem{definition}[theorem]{Definition}
\newtheorem{exmp}[theorem]{Example}
\newtheorem{remark}[theorem]{Remark}

\begin{document}
	
\maketitle

\begin{abstract}
Neural codes, represented as collections of binary strings called codewords, are used to encode neural activity. A code is called convex if its codewords are represented as an arrangement of convex open sets in Euclidean space. Previous work has focused on addressing the question: how can we tell when a neural code is convex? 
Giusti and Itskov 
identified a local obstruction and proved that convex neural codes have no local obstructions. 
The converse is true for codes on up to four neurons, but false in general.
Nevertheless, we prove that this converse holds for codes with up to three maximal codewords, and moreover the minimal embedding dimension of such codes is at most two.
\end{abstract}


\section{Introduction}
The brain encodes spatial structure through neurons in the hippocampus known as \emph{place cells}, which are associated with regions of space called \emph{receptive fields}. Place cells fire at a high rate precisely when the animal is in that receptive field. The firing pattern of these neurons form what is called a {\em neural code}. A neural code is {\em convex} if it is generated by receptive fields that are convex. 
Such convex receptive fields are observed experimentally, 
so much work has focused on understanding which neural codes are convex~\cite{ cruz2019open, curto2017, curto2013neural,  goldrup2020classification, GOY, Jeffs2018convex, KLR}.


Giusti and Itskov identified a combinatorial criterion, called 
a local obstruction, and proved that if a neural code is convex, then it has no local obstructions~\cite{giusti_itskov2014}. 
The converse is false: Lienkaemper {\em et al.} found a counterexample code on five neurons with four maximal codewords~\cite{Lienkaemper_2017}.  
For codes on up to four neurons, however, the converse is true~\cite{curto2017}.  
Similarly, we prove that 
the converse holds 
for 
codes with up to three maximal codewords, as follows.

\begin{thm} \label{thm:main}
Let $\mathcal{C}$ be a neural code with up to three maximal codewords.  Then $\mathcal{C}$ is convex if and only if $\mathcal{C}$ has no local obstructions, and additionally the minimal embedding dimension of such a code is at most two (that is, the convex receptive fields can be drawn in $\mathbb{R}^2$).
\end{thm}

Theorem~\ref{thm:main}
extends a prior result pertaining to the case of a unique maximal codeword~\cite{curto2017}. 
Another case when convexity is equivalent to having no local obstructions is when all codewords have size up to two~\cite{sparse}. 
However, for codewords of size up to three, the equivalence is false: see the codes in~\cite[\S 2.3]{cruz2019open}) and~\cite[Theorem 4.1]{goldrup2020classification}.

Our proofs are combinatorial in nature, and our construction of receptive fields in $\mathbb{R}^2$ is inspired by a similar construction in~\cite{curto2017}.  We also rely on a result of Cruz {\em et al.}, which states that max-intersection-complete codes (that is, codes that contain all possible intersections of maximal codewords) are convex~\cite{cruz2019open}.  

This work is organized as follows.  
In Section~\ref{sec:bkrd}, we recall definitions and previous results. 
We prove Theorem~\ref{thm:main} in Section~\ref{sec:results}, and
then end with a discussion in Section~\ref{sec:discussion}.

\section{Background} \label{sec:bkrd}
In this section, we introduce definitions, notations, and previous results.

\subsection{Neural Codes}
In a biological context, a codeword represents a set of neurons that fire together while no other neurons fire.  A neural code is a set of such codewords. 

\begin{definition} \label{def:code}
A \emph{neural code} $\mathcal{C}$ on $n$ neurons is a set of subsets of $[n]$ (called \emph{codewords}), i.e. $\mathcal{C} \subseteq 2^{[n]}$. A \emph{maximal codeword} of
$\mathcal{C}$ is a codeword that is not properly contained in any other codeword in $\mathcal{C}$.  
A code $\code$ is $\emph{max-intersection-complete}$ if it contains every intersection of two or more maximal codewords of $\mathcal{C}$.
\end{definition}
\begin{definition}
For a neural code $\mathcal{C}$ on \emph{n} neurons, a collection 
$\mathcal{U}= \{U_1,U_2,  \dots , U_n\}$ of subsets of a set $X$ \emph{realizes} $\mathcal{C}$ if a codeword $\sigma$ is in $\mathcal{C}$ if and only if ($\bigcap _{i\in \sigma}$$U_i$) $\smallsetminus$ $\bigcup_{i\notin \sigma}$$U_i$ is nonempty.
By convention,  $U_{\emptyset} := X$. 
\end{definition}

We will assume that all codes contain the empty set, and will always take $X=\mathbb{R}^d$, for some $d$ (see~\cite[Remark 2.19]{new-obs}).

\begin{definition} \label{def:convex}
A neural code $\code$ is \emph{convex} 
    if it can be realized by a set of convex open sets $U_1,U_2, \dots, U_n \subseteq \mathbb{R}^d.$ The smallest value of $d$ for which this is possible is the \emph{minimal embedding dimension} of $\mathcal{C}$, denoted by $\dim(\code)$.
\end{definition}

\begin{exmp} \label{ex:realize}
Consider the code $\mathcal{C} = \{ \textbf{1234}, 12,  3, 4, \emptyset \}$, where the maximal codeword is in bold. 
A convex realization $\mathcal{U}= \{U_1,U_2,U_3,U_4\}$ of this code is depicted in Figure~\ref{fig:code_1234}.
\begin{figure}[H]
    \begin{center}
        \begin{tikzpicture}[scale = .6]

        \draw (2.6,-1.5) to [bend left = 30] (0,-3) to [bend left = 30] (-2.6,-1.5) to (0, 3) -- cycle [thick, draw = green, fill = green!50, opacity = .7];
        \draw (0,3) to [bend left = 47] (3,0) to [bend left = 16] (2.6,-1.5) to (-2.6, -1.5) -- cycle [thick, draw = blue, fill = blue!30, opacity = .5];
        \draw (0,3) to [bend right = 47] (-3,0) to [bend right = 16] (-2.6,-1.5) to (2.6, -1.5) -- cycle [thick, draw = yellow, fill = yellow!30, opacity = .7];
        \draw (0,3) to [bend right = 47] (-3,0) to [bend right = 16] (-2.6,-1.5) to (2.6, -1.5) -- cycle [thick, pattern=north west lines, pattern color=red, draw = red, opacity = .7];
        
        \node at (0,0) {$1234$};
        \node at (0,-2.25) {$4$};
        \node at (-2,1) {$12$};
        \node at (2,1) {$3$};
        
        \node[draw, text width=0.11\linewidth,inner sep=2mm,align=left,
                 below left] at (8, 2.3)
                {\color{yellow} \textbf{-----} \color{black} : ${U}_1$ \par
                \color{red} \textbf{- - -} \color{black} : ${U}_2$ \par 
                \color{blue} \textbf{-----} \color{black} : ${U}_3$ \par
                \color{green} \textbf{-----} \color{black} : ${U}_4$ \par};
        \end{tikzpicture}
    \end{center}
    \caption{Convex realization of the code $\mathcal{C} = \{ \textbf{1234}, 12, 3, 4, \emptyset \}$. \label{fig:code_1234}
    }
    \end{figure}
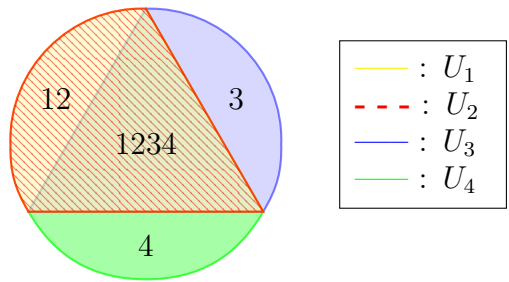
\end{exmp}

We recall the following result of Cruz {\em et al.}~\cite[Theorem 1.2]{cruz2019open}.

\begin{prop}[Max-intersection-complete $\Rightarrow$ convex] \label{prop:max-intersection-complete}
If $\mathcal{C}$ is a max-intersection-complete code with exactly $k$ maximal codewords, then $\mathcal{C}$ 
is 
convex and 
$\dim(\code) \leq \max \{2, k-1 \}$.
\end{prop}

\subsection{Simplicial Complexes}

\begin{definition} 
An abstract \emph{simplicial complex} on $n$ vertices is a nonempty set of subsets (\emph{faces}) of $[n]$ that is closed under taking subsets. \emph{Facets} are the faces of a simplicial complex that are maximal with respect to inclusion. 
\end{definition}

For a code $\mathcal{C}$ on $\emph{n}$ neurons, $\Delta$($\mathcal{C}$) is the smallest simplicial complex on $[n]$ that contains $\mathcal{C}$: 

$$\Delta(\mathcal{C}) ~:=~ \{\omega\subseteq [n] \mid \omega\subseteq\sigma \mbox{ for some } \sigma\in\mathcal{C}\}.
$$
Note that two codes on $n$ neurons have the same simplicial complex $\Delta$ if and only if they have the same maximal codewords (which are the facets of $\Delta$).

We recall the following monotonicity result of Cruz {\em et al.}, which states that adding non-maximal codewords to a code preserves convexity~\cite[Theorem~1.3]{cruz2019open}. 

\begin{prop}[Convexity is monotone] \label{prop:monotone}
Let $\code$ and $\mathcal{D}$ be neural codes such that $\code \subseteq \mathcal{D} \subseteq \Delta(\code)$. 
If $\code$ is convex, then 
$\mathcal{D}$ is also convex and $\dim(\mathcal{D}) \leq \dim(\code) +1$.
\end{prop}

\begin{definition}
For a face $\sigma\:\in\:\Delta,$ the \emph{link} of $\sigma$ in $\Delta$ is the simplicial complex: 
\[
{\rm Lk}_{\sigma}(\Delta) ~:=~ \{
     \omega \subseteq \Delta \mid \sigma \cap \omega = \emptyset ~{\rm and}~ \sigma \cup \omega \in \Delta \}~.
\]
\end{definition}

\begin{exmp} \label{ex:simplicial-cpx} 
Consider the neural code $\mathcal{C} = \{ {\bf 1356, 123, 124 }, 12, 13, 3, \emptyset \}$.  
The simplicial complex $\Delta=\Delta(\code)$ has facets $\{1356, 123, 124\}$.  Depicted in Figure~\ref{fig:s-cpx} are $\Delta$ and the link ${\rm Lk}_{\{1\}} (\Delta) $ of the triplewise intersection $1356\cap 123 \cap 124= 1$.
\begin{figure}[ht]
    \begin{center}
    \begin{tikzpicture}[scale=.5]
    \draw(-2,-1) -- (0,-2)-- (0,0)-- cycle[thick, draw = red, fill = red!30, opacity = .5];
    \draw (0,0) -- (2,-1) -- (0,-2) [draw = black, thick, fill = black!30, opacity = .7];
    \draw (0,0) -- (2,3)-- (4,0)-- (2,-1) -- (2,3) [draw = blue, thick, fill = blue!30, opacity = .3];
    \filldraw (0,0) -- (2,3) -- (2,-1) -- cycle [draw = blue, thick, fill = blue!30, opacity = .3];
    \draw [dashed](0,0) -- (4,0);
    \filldraw (0,0) -- (4,0) -- (2,-1) -- cycle [draw = blue, thick, fill = blue!30, opacity = .3];
     
    \node at (-2.5,-1) {4};   
    \node at (-.5,.5) {1};
    \node at (0,-2.5) {2};
    \node at (2,-1.5) {3};
    \node at (2,3.5) {6};
    \node at (4.5,0) {5};
    
    
    \draw(8,-1) -- (10,-2)[thick, draw = red, fill = red!30, opacity = .5];
    \draw (12,-1) -- (10,-2) [draw = black, thick, fill = black!30, opacity = .7];
    \draw (12,3)-- (14,0)-- (12,-1) -- cycle [draw = blue, thick, fill = blue!30, opacity = .3];
     
    \node at (7.5,-1) {4}; 
    \node at (10,-2.5) {2};
    \node at (12,-1.5) {3};
    \node at (12,3.5) {6};
    \node at (14.5,0) {5};
    
    \end{tikzpicture}
    \end{center}
    \caption{(Left) The simplicial complex 
    $\Delta$ with facets $\{1356, 123, 124\}$; (Right) The link ${\rm Lk}_{\{1\}} (\Delta) $. \label{fig:s-cpx}
    }
\end{figure}
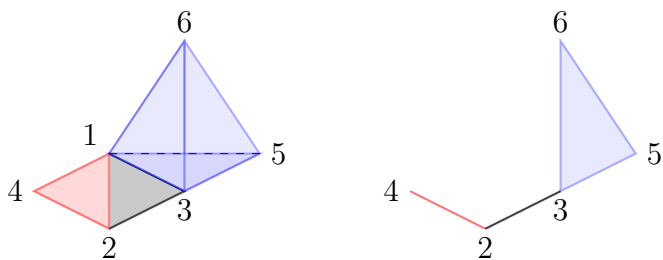
\end{exmp}

Recall that a set is \emph{contractible} if it is homotopy-equivalent to a single point. 
We see in Figure~\ref{fig:s-cpx} that ${\rm Lk}_{\{1\}} (\Delta)$ is contractible.  In this example, the codeword $1$ is the intersection of three facets.  Next, we recall 
what happens when only two facets intersect; now the link is non-contractible, as follows~\cite[Lemma 4.7]{curto2017}.

\begin{lem}
\label{lem:intersect-2}
Let $\sigma$ be a face of a simplicial complex $\Delta$. If $\sigma = \tau_1 \cap \tau_2$, where $\tau_1$ and $\tau_2$ are distinct facets of $\Delta$, and $\sigma$ is not contained in any other facet, then ${\rm Lk}_\sigma(\Delta)$ is not contractible.
\end{lem}

To state another useful lemma concerning links, we need the following definition.

\begin{definition} \label{def:nerve}
For a collection of subsets  $\mathcal{W} = \{W_1, W_2, \dots, W_n \}$ of a set ${X}$, the \emph{nerve} of $\mathcal{W}$ is the simplicial complex that records the intersection patterns among the sets: 
\[
\mathcal{N}(\mathcal{W}) := \left\{ I \subseteq [n] \mid \bigcap_{i \in I} W_i \text{ is nonempty} \right\}.
\]
\end{definition}

Lienkaemper {\em et al.}\ 
used the nerve lemma to prove the following result~\cite[Equation~(2)]{Lienkaemper_2017}.

\begin{lem} 
\label{lem:LSW}
Let $\sigma$ be a face of a simplicial complex $\Delta$, and
let $\mathcal{L}_\sigma (\Delta)$ be the set of facets of the link ${\rm Lk}_\sigma (\Delta)$:
\[
	\mathcal{L}_{\sigma} (\Delta) ~=~ \{ (M \smallsetminus \sigma) \mid M \textrm{ is a facet of } \Delta \textrm{ that contains } \sigma \}~.
\]
Then the following homotopy-equivalence holds: ${\rm Lk}_\sigma(\Delta) \simeq \mathcal{N}(\mathcal{L}_\sigma(\Delta))$. 
\end{lem}

We will use Lemma~\ref{lem:LSW} 
to analyze the case when $\sigma$ is the intersection of three facets.

\subsection{Local Obstructions}

The following definition is equivalent to the standard one~\cite{curto2017}.

\begin{definition} \label{def:local-obs}
A neural code $\code$ with simplicial complex $\Delta = \Delta(\mathcal{C})$ has a \emph{local obstruction} if there exists a nonempty face $\sigma \in \Delta$ such that the following hold:
\begin{enumerate}
    \item 
    $\sigma$ is the intersection of two or more facets of $\Delta$, 
    \item 
    the link ${\rm Lk}_\sigma(\Delta)$ is not contractible, 
and 
    \item 
    $\sigma \notin \code$.
\end{enumerate}
\end{definition}

The following result is due to Giusti and Itskov~\cite{giusti_itskov2014}.
\begin{prop} 
\label{prop:local-obs}
Every convex neural code has no local obstructions.
\end{prop}

The \emph{minimal code} of a simplicial complex $\Delta$, denoted by $\code_{\rm min}(\Delta)$,
consists of: 
\begin{enumerate}
    \item 
    all facets of $\Delta$, 
    \item 
    all faces $\sigma\in \Delta$ that are the 
 intersection of two or more facets of $\Delta$, such that ${\rm Lk}_\sigma(\Delta)$ is not contractible, and 
    \item 
 the empty set.  
\end{enumerate}
By Proposition~\ref{prop:local-obs}, 
 $\code_{\rm min}(\Delta)$ 
 is the unique minimal (with respect to inclusion) code among all codes with simplicial complex $\Delta$ and no local obstructions.

\begin{exmp} \label{ex:min-code}
For the simplicial complex~$\Delta$ in Example~\ref{ex:simplicial-cpx}, 
the minimal code is $\code_{\rm min}(\Delta) = \{ {\bf 1356, 123, 124 }, 12, 13, \emptyset \}$.  
\end{exmp}

As detailed in the Introduction, the converse of Proposition~\ref{prop:local-obs} is false~\cite{Lienkaemper_2017}, but true in some cases, such as when all codewords have size at most two~\cite{sparse} or for codes on up to four neurons~\cite{curto2017}.  
Our main result, Theorem~\ref{thm:main}, gives another such class, a special case of which is as follows.
\begin{prop}[Convexity for codes with up to two maximal codewords] \label{prop:2-max}
Assume $\mathcal{C}$ is a neural code with exactly one or two maximal codewords. Then 
the following are equivalent:
\begin{itemize}
    \item $\code$ is convex,
    \item $\code$ has no local obstructions, and
    \item $\code$ is max-intersection-complete.
\end{itemize}
Also, if $\code$ is convex, then $\dim(\code) \leq 2$.
\end{prop}
\begin{proof}
This result follows easily from Lemma~\ref{lem:intersect-2},  Propositions~\ref{prop:max-intersection-complete} and~\ref{prop:local-obs}, and Definition~\ref{def:local-obs}.
\end{proof}

\section{Main Results} \label{sec:results}
The aim of this section is to prove our main result on codes with up to three maximal codewords (Theorem~\ref{thm:main}).  Our proof requires two preliminary results (Lemmas~\ref{lem:contractible} and~\ref{lem:construction}), 
which pertain to simplicial complexes like the one in Example~\ref{ex:simplicial-cpx}.  Specifically, the three facets of the simplicial complex are arranged in a certain way, which we now define.

Let $\Delta$ be a simplicial complex with exactly three facets $F_1$, $F_2$, and $F_3$.  We say that $\Delta$ satisfies the 
{\em Path-of-Facets Condition}
if 
exactly one of the following three sets is empty (and the other two are nonempty): 
$( F_1 \cap F_2 ) \smallsetminus F_3$,
$( F_1 \cap F_3 ) \smallsetminus F_2$, and
$( F_2 \cap F_3 ) \smallsetminus F_1$. 
The reason behind the name of this condition is shown in the proof of the following lemma.

\begin{lem} \label{lem:contractible}
Let $\Delta$ be a simplicial complex with exactly three facets $F_1$, $F_2$, and $F_3$.  
Then 
${\rm Lk}_{F_1 \cap F_2 \cap F_3} (\Delta)$ is contractible
 if and only if $\Delta$ satisfies the Path-of-Facets Condition.
\end{lem}
\begin{proof}
Let $\sigma = F_1 \cap F_2 \cap F_3$.
By Lemma~\ref{lem:LSW}, the link  ${\rm Lk}_\sigma(\Delta)$ 
is homotopy-equivalent to the nerve of 
$	\mathcal{L}_{\sigma} (\Delta) = \{ (F_1 \smallsetminus \sigma), ~(F_2 \smallsetminus \sigma), ~(F_3 \smallsetminus \sigma)\} $.  This nerve does not contain a $2$-simplex (i.e., a filled-in triangle) because the triplewise intersection $\bigcap_{i=1}^3 (F_i \smallsetminus \sigma) $ is empty.  So, the nerve is a graph on three vertices.  
The only such graph that is contractible is the path: 
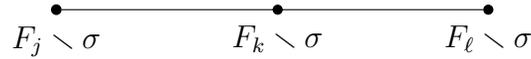
\begin{center}
    \begin{tikzpicture}[scale=.7]
    \draw (0,0) node[anchor = east, circle, fill, inner sep = 0pt, minimum size = .1 pt,label = below: $F_j \smallsetminus \sigma$, scale = .3]{$a^*$} -- (4,0) node[anchor = west,circle,fill, inner sep = 0pt, minimum size = .1 pt,label = below: $F_k \smallsetminus \sigma$, scale = .3]{$b^*$} -- (8,0) node[anchor = west,circle,fill, inner sep = 0pt, minimum size = .1 pt,label = below: $F_{\ell} \smallsetminus \sigma$, 
    scale = .3]{$b^*$};
    \end{tikzpicture}
\end{center}
(Here, $j,k, \ell$ form a permutation of $1,2,3$.)
We conclude that  ${\rm Lk}_\sigma(\Delta)$ is contractible if and only if $( F_j \cap F_k ) \smallsetminus \sigma \neq \emptyset$,
$( F_k \cap F_{\ell} ) \smallsetminus \sigma \neq \emptyset$, and 
$( F_j \cap F_{\ell} ) \smallsetminus \sigma = \emptyset$ (for some permutation $j,k, \ell$ of $1,2,3$), 
which is 
easily seen to be equivalent to the Path-of-Facets Condition. 
\end{proof}

The next result, Lemma~\ref{lem:construction}, states that,
when the Path-of-Facets Condition holds, 
the minimal code can be realized in $\mathbb{R}$ by convex open sets (i.e., open intervals).  
The proof constructs such a realization, which we illustrate in the following example.

\begin{exmp} \label{ex:1-d}  
Recall from Example~\ref{ex:min-code} that for the simplicial complex $\Delta$ with facets $\{1356, 123, 124 \}$, 
the minimal code is  $\code_{\rm min}(\Delta) = \{ {\bf 1356, 123, 124 }, 12, 13, \emptyset \}$.  
A $1$-dimensional convex realization $\mathcal{U}= \{U_1,U_2,\dots, U_6 \}$ is shown in Figure~\ref{fig:rec-fields}, and the regions defined by this realization are labeled by the corresponding codewords in Figure~\ref{fig:1-d-realization}.

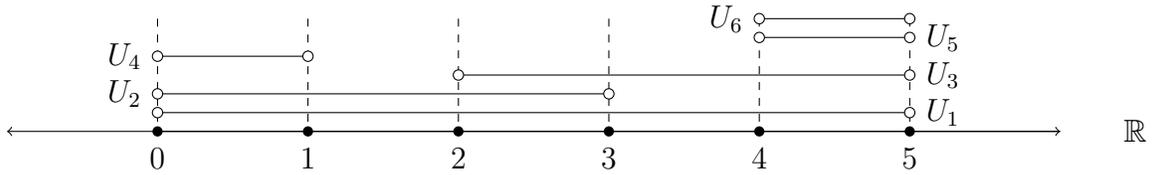
\begin{figure}[ht]
    \centering
    \begin{tikzpicture}[scale = 0.5]
        \draw (0,0) node[circle, fill, inner sep = 0pt, minimum size = .1 pt,label = below: $0$, scale = .4]{0} -- 
        (4,0) node[circle,fill, inner sep = 0pt, minimum size = .1 pt,label = below: 1, scale = .4]{1} --
        (8,0) node[circle,fill, inner sep = 0pt, minimum size = .1 pt,label = below: 2, scale = .4]{2} --
        (12,0) node[circle,fill, inner sep = 0pt, minimum size = .1 pt,label = below: $3$, scale = .4]{3} -- (16,0) node[circle,fill, inner sep = 0pt, minimum size = .1 pt,label = below: $4$, scale = .4]{4} -- 
        (20,0) node[circle,fill, inner sep = 0pt, minimum size = .1 pt,label = below: $5$, scale = .4]{5} -- (24,0) ;
        \draw (26, 1) node[circle, fill = white, label = below: $\mathbb{R}$]{};
        \draw (0,3) -- (0,0) [dashed];
        \draw (4,3) -- (4,0) [dashed];
        \draw (8,3) -- (8,0) [dashed];
        \draw (12,3) -- (12,0) [dashed];
        \draw (16,3) -- (16,0) [dashed];
        \draw (20,3) -- (20,0) [dashed];
        %
       \draw (0,.5) -- (20,.5);
        \draw (0,.5) node[circle,draw=black, fill=white, inner sep = 0pt, minimum size = 0.1 pt, scale = .5]{$~~$};
        \draw (20,.5) node[circle,draw=black, fill=white, inner sep = 0pt, minimum size = 0.1 pt, scale = .5, label = east: ${U}_1$]{$~~$};
        \draw (0,1) -- (12,1);
        \draw (0,1) node[circle,draw=black, fill=white, inner sep = 0pt, minimum size = 0.1 pt, scale = .5, label = west: ${U}_2$]{$~~$};
        \draw (12,1) node[circle,draw=black, fill=white, inner sep = 0pt, minimum size = 0.1 pt, scale = .5]{$~~$};
        \draw (8,1.5) -- (20,1.5);
        \draw (8,1.5) node[circle,draw=black, fill=white, inner sep = 0pt, minimum size = 0.1 pt, scale = .5]{$~~$};
        \draw (20,1.5) node[circle,draw=black, fill=white, inner sep = 0pt, minimum size = 0.1 pt, scale = .5, label = east: ${U}_3$]{$~~$};
        \draw (0,2) -- (4,2);
        \draw (0,2) node[circle,draw=black, fill=white, inner sep = 0pt, minimum size = 0.1 pt, scale = .5, label = west: ${U}_4$]{$~~$};
        \draw (4,2) node[circle,draw=black, fill=white, inner sep = 0pt, minimum size = 0.1 pt, scale = .5]{$~~$};
        \draw (16,2.5) -- (20,2.5);
        \draw (16,2.5) node[circle,draw=black, fill=white, inner sep = 0pt, minimum size = 0.1 pt, scale = .5]{$~~$};
        \draw (20,2.5) node[circle,draw=black, fill=white, inner sep = 0pt, minimum size = 0.1 pt, scale = .5, label = east: ${U}_5$]{$~~$};
        \draw (16,3) -- (20,3);
        \draw (16,3) node[circle,draw=black, fill=white, inner sep = 0pt, minimum size = 0.1 pt, scale = .5, label = west: ${U}_6$]{$~~$};
        \draw (20,3) node[circle, draw=black, fill=white, inner sep = 0pt, minimum size = 0.1 pt, scale = .5]{$~~$};
        \draw[<->](-4,0) -- (24,0);
     \end{tikzpicture}
    \caption{Convex realization of $\{ {\bf 1356, 123, 124 }, 12, 13, \emptyset \}$ in $\mathbb{R}$.}
    \label{fig:rec-fields}
    \end{figure}
    
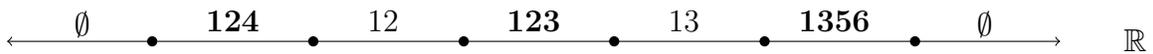
\begin{figure}[ht]
    \centering
\begin{tikzpicture}[scale=.5]
\draw (0,0) node[anchor = east, circle, fill, inner sep = 0pt, minimum size = .1 pt, 
    scale = .4]{0} -- 
(4,0) node[anchor = west,circle,fill, inner sep = 0pt, minimum size = .1 pt, 
    scale = .4]{1} -- 
(8,0) node[anchor = west,circle,fill, inner sep = 0pt, minimum size = .1 pt, 
    scale = .4]{2} -- 
(12,0) node[anchor = west,circle,fill, inner sep = 0pt, minimum size = .1 pt, 
    scale = .4]{3} -- 
(16,0) node[anchor = west,circle,fill, inner sep = 0pt, minimum size = .1 pt, 
    scale = .4]{4} -- 
(20,0) node[anchor = west,circle,fill, inner sep = 0pt, minimum size = .1 pt, 
    scale = .4]{5} ;
\draw (2,1.5) node[circle, fill = white, label = below: ${\bf 124}$]{};
\draw (6,1.5) node[circle, fill = white, label = below: $12$]{};
\draw (10,1.5) node[circle, fill = white, label = below: ${\bf 123}$]{};
\draw (14,1.5) node[circle, fill = white, label = below: $13$]{};
\draw (18,1.5) node[circle, fill = white, label = below: ${\bf 1356}$]{};
\draw (22,1.5) node[circle, fill = white, label = below: $\emptyset$]{};
\draw (-2, 1.5) node[circle, fill = white, label = below: $\emptyset$]{};
\draw (26, 1) node[circle, fill = white, label = below: $\mathbb{R}$]{};

\draw[->] (0,0) -- (-4,0);
\draw[->] (20,0) -- (24,0);
\end{tikzpicture}
    \caption{The realization in Figure~\ref{fig:rec-fields}, with regions labeled by codewords.}
    \label{fig:1-d-realization}
\end{figure}

\end{exmp}

\begin{lem} \label{lem:construction}
Let $\Delta$ be a simplicial complex with exactly three facets $F_1$, $F_2$, and $F_3$.  Assume that $\Delta$ satisfies the Path-of-Facets Condition.  
Then the minimal code $\code_{\rm min}(\Delta)$ is convex, and 
moreover
$\dim( \code_{\rm min}(\Delta) ) = 1$.
\end{lem}

\begin{proof}
As 
$\Delta$ satisfies the Path-of-Facets Condition, 
we can relabel the facets by $F_a,F_b,F_c$ so that 
$(F_a \cap F_b ) \smallsetminus F_c$ 
and 
$(F_b \cap F_c ) \smallsetminus F_a$ are nonempty, 
while
$(F_a \cap F_c ) \smallsetminus F_b$ is empty (that is, $F_a \cap F_c = F_a \cap F_b \cap F_c$).  
Next, the link ${\rm Lk}_{F_1 \cap F_2 \cap F_3} (\Delta)$ is contractible (by Lemma~\ref{lem:contractible}), so $F_a \cap F_b \cap F_c \notin \code_{\rm min}(\Delta)$.  On the other hand, by a straightforward application of Lemma~\ref{lem:intersect-2}, the links of $F_a \cap F_b$ and $F_b \cap F_c$ are both not contractible.  
We conclude that the minimal code is 
$\code_{\rm min}(\Delta) = 
\{ \mathbf{
F_a, ~F_b, ~F_c, }~
F_a \cap F_b, ~F_b \cap F_c,~ \emptyset \}$. 

We will construct a convex open realization $\mathcal{U} = \{U_i\}_i$ of $\code_{\rm min}(\Delta)$ in $\mathbb{R}$ such that the codewords appear in the order depicted in Figure~\ref{fig:realization} (which generalizes  Figure~\ref{fig:1-d-realization}).  Accordingly, for each neuron $i$, we define the receptive field $U_i$ as follows:

\begin{equation*}
    U_i ~:=~ \begin{cases}
               (0,5)               & \textrm{if } i \in F_a \cap F_b \cap F_c\\
                (0,3)              & \textrm{if } i \in (F_a \cap F_b) \smallsetminus F_c\\
                (2,5)              & \textrm{if } i \in (F_b \cap F_c) \smallsetminus F_a\\
                (0,1)              & \textrm{if } i \in F_a \smallsetminus (F_b \cup F_c) \\
                (2,3)              & \textrm{if } i \in F_b \smallsetminus (F_a \cup F_c) \\
                (4,5)              & \textrm{if } i \in F_c \smallsetminus (F_a \cup F_b) \\
           \end{cases}
\end{equation*}

\begin{figure}[ht]
    \centering
\begin{tikzpicture}[scale=.5]

\draw (0,0) node[anchor = east, circle, fill, inner sep = 0pt, minimum size = .1 pt,label = below: $0$, scale = .4]{0} -- 
(4,0) node[anchor = west,circle,fill, inner sep = 0pt, minimum size = .1 pt,label = below: 1, scale = .4]{1} -- 
(8,0) node[anchor = west,circle,fill, inner sep = 0pt, minimum size = .1 pt,label = below: 2, scale = .4]{2} -- 
(12,0) node[anchor = west,circle,fill, inner sep = 0pt, minimum size = .1 pt,label = below: $3$, scale = .4]{3} -- 
(16,0) node[anchor = west,circle,fill, inner sep = 0pt, minimum size = .1 pt,label = below: $4$, scale = .4]{4} -- 
(20,0) node[anchor = west,circle,fill, inner sep = 0pt, minimum size = .1 pt,label = below: $5$, scale = .4]{5}; 

\draw[->] (0,0) -- (-4,0);
\draw[->] (20,0) -- (24,0);

\draw (2,1.5) node[circle, fill = white, label = below: $\mathbf{F_a}$]{};
\draw (6,1.5) node[circle, fill = white, label = below: $F_a \cap F_b$]{};
\draw (10,1.5) node[circle, fill = white, label = below: $\mathbf{F_b}$]{};
\draw (14,1.5) node[circle, fill = white, label = below: $F_b \cap F_c$]{};
\draw (18,1.5) node[circle, fill = white, label = below: $\mathbf{F_c}$]{};
\draw (22,1.5) node[circle, fill = white, label = below: $\emptyset$]{};
\draw (-2, 1.5) node[circle, fill = white, label = below: $\emptyset$]{};
\draw (26, .5) node[circle, fill = white, label = below: $\mathbb{R}$]{};
\end{tikzpicture}
    \caption{Convex realization of $\mathcal{C}_{min}(\Delta)$ in $\mathbb{R}$.}
    \label{fig:realization}
\end{figure}
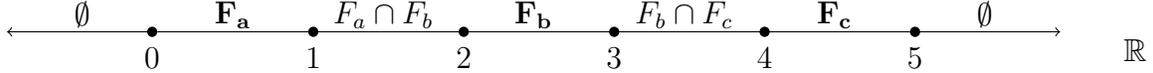


We now show that the realization of $\mathcal{U}= \{U_i\}_i$ contains all codewords in $\mathcal{C}_{\rm min}(\Delta)$ and no other codewords. 
It is evident, by construction, that each codeword of $\mathcal{C}_{\rm min}(\Delta)$ appears in the intervals indicated in Figure~\ref{fig:realization} (e.g., the codeword $F_a$ is realized in the interval $(0,1)$).
%
%
So, all that is left to show is that no additional codewords are realized at the endpoints of the intervals.
%
%
Indeed, it is straightforward to check that the endpoints $0,1,2,3,4,5$ give rise to the codewords $\emptyset$, $F_a \cap F_b$, $F_a \cap F_b$, $F_b \cap F_c$, $F_b \cap F_c,$ and $\emptyset$, respectively. 
\end{proof}

We can now completely characterize convexity of codes $\mathcal{C}$ with three maximal codewords. 
In what follows (see the proof of Theorem~\ref{thm:main-restated}), we show that when 
$\Delta(C)$ does not satisfy the Path-of-Facets Condition, we have:

\begin{center}
max-intersection-complete $\Leftrightarrow$ convex $\Leftrightarrow$ no local obstructions.    
\end{center}
On the other hand, when 
$\Delta(C)$ does satisfy the Path-of-Facets Condition, we have only convex $\Leftrightarrow$ no local obstructions. 

For convenience, we restate Theorem~\ref{thm:main}, as follows.

\begin{thm}[Theorem~\ref{thm:main}, restated] \label{thm:main-restated}
If $\mathcal{C}$ is a neural code with up to three maximal codewords, then:
\begin{itemize}
    \item $\mathcal{C}$ is convex if and only if $\mathcal{C}$ has no local obstructions, and
    \item if $\mathcal{C}$ is convex, then $\dim(\code) \leq 2.$
\end{itemize}
\end{thm}

\begin{proof} 
We already know that convex codes have no local obstructions (Proposition~\ref{prop:local-obs}).  
For the converse, let $\mathcal{C}$ be a code with up to three maximal codewords and no local obstructions.  We must show that $\mathcal{C}$ is convex and, moreover, $\dim(\code) \leq 2$.  

The case of one or two maximal codewords is Proposition~\ref{prop:2-max}.  
So, now assume that $\mathcal{C}$ has exactly three maximal codewords.  We first consider the subcase when $\Delta(\mathcal{C})$ satisfies the Path-of-Facets Condition.  We have that $\mathcal{C}_{\rm min}(\Delta(\mathcal{C})) \subseteq \mathcal{C} \subseteq \Delta(\mathcal{C})$ (and $\Delta(\mathcal{C})$ is the simplicial complex of $\mathcal{C}_{\rm min}(\Delta(\mathcal{C}))$).
Thus, Lemma~\ref{lem:construction} and Proposition~\ref{prop:monotone} 
together imply that $\code$ is convex and
 $\dim(\code) \leq 2$.

We consider the remaining subcase, when $\Delta(\mathcal{C})$ does not satisfy the Path-of-Facets Condition.  We first claim that $\mathcal{C}$ is max-intersection-complete.  
To see this, let $\sigma$ be the intersection of two or three maximal codewords of $\code$.  If $\sigma$ is the intersection of two maximal codewords and is not contained in the third maximal codeword, then Lemma~\ref{lem:intersect-2} implies that ${\rm Lk}_\sigma(\Delta)$ is not contractible and so (as $\code$ has no local obstructions) $\sigma \in \code$.  
If, on the other hand, 
$\sigma$ is the intersection of all three maximal codewords, then, by Lemma~\ref{lem:contractible}, 
${\rm Lk}_\sigma(\Delta)$ is again not contractible (because the Path-of-Facets Condition does not hold) 
and so (as before) $\sigma \in \code$.  Hence, our claim holds and 
so, Proposition~\ref{prop:max-intersection-complete} implies that $\code$ is convex with $\dim(\code) \leq 2$.
\end{proof}

\begin{exmp}  \label{ex:2-d}
Recall from Example~\ref{ex:1-d} that the following code has minimal embedding dimension 1: 
$\code_{\rm min}(\Delta) = \{ {\bf 1356, 123, 124 }, 12, 13, \emptyset \}$.  
By adding the non-maximal codewords $\{23, 24, 5, 6\}$, we obtain a code which we denote by $\mathcal{D}$.
A $2$-dimensional convex realization of $\mathcal{D}$ is depicted in Figure~\ref{fig:2-d-realization}.  This realization is obtained by 
following the proof of Theorem~\ref{thm:main} 
(which, via Proposition~\ref{prop:monotone}, relies on a construction of Cruz {\em et al.}~\cite{cruz2019open}).
Informally, the steps are as follows.  Starting from the $1$-dimensional realization 
of $\mathcal{C}_{\rm min}(\Delta(\mathcal{C}))$ 
in Figure~\ref{fig:realization}, we  ``fatten'' each interval $U_i$ into $\mathbb{R}^2$ and then intersect with an open ball.  Then, for each additional non-maximal codeword $\tau$, we ``slice off'' part of a region corresponding to a codeword $\widetilde \tau$ for which $\tau \subseteq \widetilde \tau$.

\begin{figure}[ht]
    \centering
    \begin{tikzpicture}
                \draw (-3,0) -- (-1.8,2.4) to [bend left = 10] (-.6,2.922) -- (.6,2.922) to [bend left = 10] (1.8,2.4) -- (3,0) -- (1.8,-2.4) to [bend left = 33] (-1.8,-2.4) to [bend left = 25] (-3,0) [fill = gray, draw = gray, opacity = .5];
                
                \draw (-.6,2.922) to [bend left = 23] (1.8,2.4) -- (3,0) -- (1.8,-2.4) to [bend left = 23] (-.6,-2.922) -- (-.6, 2.922) [draw = green, thick, fill = green!30, opacity = .5];
                
                \draw (.6, -2.922) to [bend left = 21] (-1.8,-2.4) to [bend left = 25] (-3,0) to [bend left = 27] (-1.8,2.4) to [bend left = 21] (.6,2.922) -- (.6,-2.922) [thick, draw = red, pattern = horizontal lines, pattern color = red, opacity = .75];
                
                \draw (-1.8,-2.4) to [bend left = 58] (-1.8,2.4) -- (-1.8,-2.4)[draw = blue, thick, fill = blue!20, opacity = .5];
                
                \draw (1.8,2.4) to [bend left = 25] (3,0) -- (1.8,-2.4) -- (1.8,2.4) [thick, draw = yellow, fill = yellow!30, opacity = .7];
                
                \draw (1.8,2.4) -- (3,0) to [bend left = 25] (1.8,-2.4) -- (1.8,2.4) [thick, draw = orange, pattern = horizontal lines, pattern color = orange, opacity = .75];
                
                \draw (-3,0) -- (-1.8,2.4) [draw = gray, opacity = .5];
                \draw (1.8,2.4) -- (3,0) [draw = orange, opacity = .5];
                \draw (.6,2.922) -- (.6,-2.922) [draw = red, opacity = .5];
                \draw (-.6,2.922) -- (.6,2.922) [draw = black, opacity = .5];
                
                \node at (-2.4,0) {$124$};
                \node at (0,0) {$123$};
                \node at (2.4,0) {$1356$};
                \node at (-1.2, 0) {$12$};
                \node at (1.2,0) {$13$};
                \node at (-3.3, 1.3) {$24$};
                \draw (-3.15,1.3) -- (-2.6,1.3) [->];
                \node at (2.5, 1.3) {$6$};
                \node at (2.5, -1.3) {$5$};
                \node at (0,3.5) {$23$};
                \draw (0,3.35) -- (0,2.95) [->];

                \node[draw, text width=0.11\linewidth,inner sep=2mm,align=left,
                 below left] at (8, 2.3)
                {\color{gray} \textbf{-----} \color{black} : ${U}_1$ \par 
                \color{red} \textbf{- - -} \color{black} : ${U}_2$ \par 
                \color{green} \textbf{-----} \color{black} : ${U}_3$ \par
                \color{blue} \textbf{-----} \color{black} : ${U}_4$ \par
                \color{orange} \textbf{- - -} \color{black} : ${U}_5$ \par
                \color{yellow} \textbf{-----} \color{black} : ${U}_6$ \par
                };
    \end{tikzpicture}
    \caption{Convex realization of $\mathcal{D} = 
    \{ {\bf 1356, 123, 124 }, 12, 13, 23, 24, 5, 6, \emptyset \}$.}
    \label{fig:2-d-realization}
\end{figure}
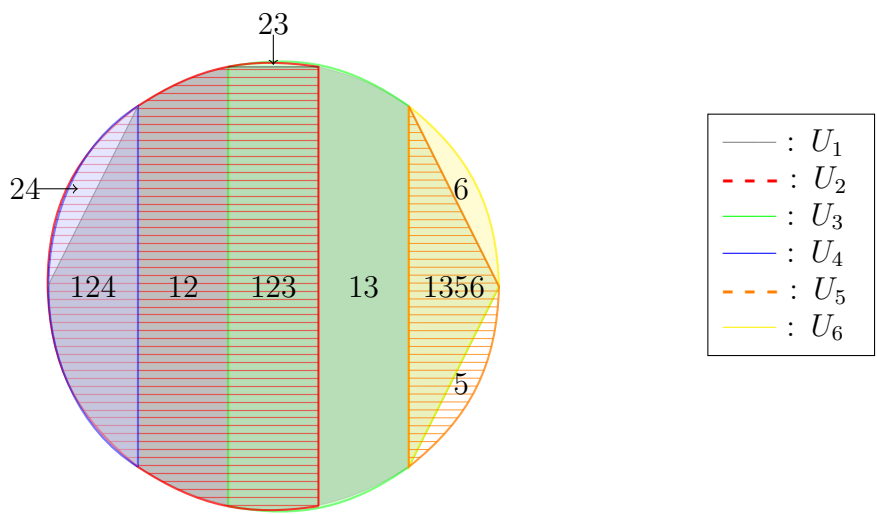
\end{exmp}

\begin{remark}[1-dimensional vs.\ 2-dimensional codes] \label{rem:1-d-vs-2-d} 
Theorem~\ref{thm:main-restated} states that convex codes with up to three maximal codewords have minimal embedding dimension one or two.  
To distinguish between these two possible dimensions, 
we refer the reader to the classification of 1-dimensional codes due to Rosen and Yan~\cite{zvi-yan}.
\end{remark}

\begin{remark}[Open convex vs.\ closed convex codes] \label{rem:open-v-closed}
Convexity in
Theorem~\ref{thm:main-restated} can be replaced by ``closed convexity'' (having a realization by convex sets that are closed).  Indeed, the realizations by convex, open sets that are used in our proofs are easily seen to be ``nondegenerate'', as defined by Cruz {\em et al.}, and so their results imply that taking closures of the open sets in a realization yields a closed, convex realization~\cite[Theorem 2.12]{cruz2019open}.
\end{remark}

\section{Discussion}\label{sec:discussion}
It is in general difficult to determine whether a given code is convex.  Nevertheless, here we showed that this task is easy for codes that have at most three maximal codewords: Convexity for such codes is equivalent to lacking local obstructions.  Also, in this setting, open convexity and closed convexity are equivalent (recall Remark~\ref{rem:open-v-closed}).  We note that neither of these equivalences holds in general for codes with four or more maximal codewords~\cite{cruz2019open,Lienkaemper_2017}.  

It is also usually difficult to ascertain the minimal embedding dimension of a convex code.  In fact, there are few results in this direction, and many such results only bound the dimension (for instance, see~\cite{curto2017,leray,GOY}).  It is therefore notable that we are able to achieve precise dimensions for a family of codes (by Theorem~\ref{thm:main-restated} and Remark~\ref{rem:1-d-vs-2-d}).  Indeed, our results help clarify which neural codes are easy to understand and which ones remain challenging.

 
\subsection*{Acknowledgements}
The authors thank Alexander Ruys de Perez for insightful discussions and useful suggestions. 
KJ and CS conducted this research in the 2020 REU in the Department of Mathematics at Texas A\&M
University, supported by NSF grant DMS-1757872.  
AS was supported by NSF grant DMS-1752672.  


\bibliography{mybibliography}{}

\begin{thebibliography}{10}

\bibitem{new-obs}
Aaron Chen, Florian Frick, and Anne Shiu.
\newblock Neural codes, decidability, and a new local obstruction to convexity.
\newblock {\em SIAM J.\ Applied Algebra and Geometry}, 3(1):44--66, 2019.

\bibitem{cruz2019open}
Joshua Cruz, Chad Giusti, Vladimir Itskov, and Bill Kronholm.
\newblock On open and closed convex codes.
\newblock {\em Discrete \& Computational Geometry}, 61(2):247--270, 2019.

\bibitem{curto2017}
Carina Curto, Elizabeth Gross, Jack Jeffries, Katherine Morrison, Mohamed Omar,
  Zvi Rosen, Anne Shiu, and Nora Youngs.
\newblock What makes a neural code convex?
\newblock {\em SIAM Journal on Applied Algebra and Geometry}, 1(1):222--238,
  2017.

\bibitem{curto2013neural}
Carina Curto, Vladimir Itskov, Alan Veliz-Cuba, and Nora Youngs.
\newblock The neural ring: an algebraic tool for analyzing the intrinsic
  structure of neural codes.
\newblock {\em Bull.\ Math.\ Biol.}, 75(9):1571--1611, 2013.

\bibitem{leray}
Carina Curto and Ram{\'o}n Vera.
\newblock The {L}eray dimension of a convex code.
\newblock {\em Available at {\tt arXiv:1612.07797}}, 2016.

\bibitem{giusti_itskov2014}
Chad Giusti and Vladimir Itskov.
\newblock A no-go theorem for one-layer feedforward networks.
\newblock {\em Neural Computation}, 26(11):2527--2540, 2014.
\newblock PMID: 25149704.

\bibitem{goldrup2020classification}
Sarah~Ayman Goldrup and Kaitlyn Phillipson.
\newblock Classification of open and closed convex codes on five neurons.
\newblock {\em Adv.\ Appl.\ Math.}, 112:101948, 2020.

\bibitem{GOY}
Elizabeth Gross, Nida Obatake, and Nora Youngs.
\newblock Neural ideals and stimulus space visualization.
\newblock {\em Adv.\ Appl.\ Math.}, 95:65--95, 2018.

\bibitem{Jeffs2018convex}
R.~Amzi Jeffs and Isabella Novik.
\newblock Convex union representability and convex codes.
\newblock {\em Int.\ Math.\ Res.\ Notices}, Apr 2019.

\bibitem{sparse}
R.~Amzi Jeffs, Mohamed Omar, Natchanon Suaysom, Aleina Wachtel, and Nora
  Youngs.
\newblock Sparse neural codes and convexity.
\newblock {\em Involve, a Journal of Mathematics}, 12(5):737--754, 2019.

\bibitem{KLR}
Alexander Kunin, Caitlin Lienkaemper, and Zvi Rosen.
\newblock Oriented matroids and combinatorial neural codes.
\newblock {\em Available at {\tt arXiv:2002.03542}}, 2020.

\bibitem{Lienkaemper_2017}
Caitlin Lienkaemper, Anne Shiu, and Zev Woodstock.
\newblock Obstructions to convexity in neural codes.
\newblock {\em Adv.\ Appl.\ Math.}, 85:31--59, 2017.

\bibitem{zvi-yan}
Zvi Rosen and Yan~X Zhang.
\newblock Convex neural codes in dimension 1.
\newblock Available at {\tt arXiv:1702.06907}, 2017.

\end{thebibliography}
\bibliographystyle{plain}

\end{document}